\documentclass{amsart}
\usepackage{amsmath,amssymb,amsthm,multirow,graphicx}

\newtheorem{theorem}{Theorem}
\newtheorem*{theorem*}{Theorem}
\newtheorem{lemma}{Lemma}

\newtheorem*{corollary*}{Corollary}

\newcommand{\fS}{\mathfrak{S}}

\newcommand{\cA}{\mathcal{A}}

\begin{document}
\title{Sifting Limits for the $\Lambda^{2}\Lambda^{-}$ Sieve}
\author 
{C.S. Franze}
\date{\today}
\begin{abstract}
Sifting limits for the $\Lambda^{2}\Lambda^{-}$ sieve, Selberg's lower bound sieve, are computed for integral dimensions $1<\kappa\le10$. The evidence strongly suggests that for all $\kappa\ge3$ the $\Lambda^{2}\Lambda^{-}$ sieve is superior to the competing combinatorial sieves of Diamond, Halberstam, and Richert. A method initiated by Grupp and Richert for computing sieve functions for integral $\kappa$ is also outlined. 
\end{abstract}
\maketitle
\section{Introduction}
Let $\mathcal{A}$ be a sequence of integers, and $\mathcal{P}$ a set of primes. Recall that the goal of the sieve method is to obtain bounds for 
\begin{equation}\label{SAP}
S(\mathcal{A},\mathcal{P},z)=\sum_{\substack{n\in\mathcal{A}\\(n,P(z))=1}}1,
\end{equation}
where
$$P(z)=\prod_{\substack{p\in\mathcal{P}\\p<z}}p.$$ 
One expects that the sequence $\mathcal{A}$ is well-behaved in that $\mathcal{A}_{d}$, the elements of $\mathcal{A}$ divisible by $d$, satisfy
\begin{equation}\label{seqA}
\left|\mathcal{A}_{d}\right|=\frac{X}{f(d)}+\mathcal{R}_{d},
\end{equation}
where $f(d)$ is some multiplicative function, and the errors, $\mathcal{R}_{d}$, are relatively small, at least on average. In fact, suppose that there exists a constant $A\ge1$ such that
\begin{equation}\label{remainder}
\sum_{d<\frac{X}{\log^{A} X}}\mu^{2}(d)7^{\nu(d)}\left|\mathcal{R}_{d}\right|\ll\frac{X}{\log^{\kappa+1}X}.
\end{equation}
In addition, one assumes that
\begin{equation}\label{loghypoth}
\sum_{p<s}\frac{\log p}{f(p)}=\kappa\log s+O\left(1\right),
\end{equation}
and refers to $\kappa$ as the \textit{dimension}, or \textit{density}, of the sieve.\\

H. Diamond and H. Halberstam, in association with the late professor H.-E. Richert, constructed a class of sieves for all dimensions $\kappa\ge1$. Their sieves (DHR sieves for short) combine elements of Selberg's $\Lambda^2$ upper bound sieve and the combinatorial sieves of Rosser-Iwaniec. For an account of their work, we refer the reader to their recent book \cite{DHG08}. An important parameter in a sieve is the \textit{sifting limit} $\beta_\kappa$, beyond which the lower bound sieve yields a positive lower bound. The calculations in Chapter 17 of \cite{DHG08} show that for the DHR sieves, $\beta_\kappa \lesssim 2.44 \kappa$. 

Selberg investigated an alternative lower bound sieve method, known as the $\Lambda^2 \Lambda^-$ sieve, for large dimensions $\kappa$. The starting point for this sieve, similar to the $\Lambda^2$ upper bound sieve, is the observation that for any set of real numbers $\lambda_d$, normalized so that $\lambda_{1}=1$,
\begin{equation*}
S(\mathcal{A}, \mathcal{P}, z) \ge 
  \sum_{n\in \cA}
        \left( 1- \sum_{\substack{p|n \\ p < z }} 1 \right)\left( \sum_{\substack{d|n\\d|P(z)}} \lambda_d \right)^2.
\end{equation*}
Selberg proved that for sufficiently large $\kappa$, this sieve yields $\beta_\kappa \lesssim 2\kappa + 19/36$. As a consequence, the $\Lambda^2 \Lambda^-$ sieve is superior to the DHR sieves if $\kappa$ is taken sufficiently large. How large is sufficiently large? For small integer $\kappa$ with $2\le\kappa\le10$, we prove
\begin{theorem}\label{T:kappabiggerthan3}
Suppose $S\left(\mathcal{A},\mathcal{P},z\right)$ is as defined in \eqref{SAP}, and that $\mathcal{A}$ satisfies \eqref{seqA}, \eqref{remainder}, and \eqref{loghypoth}. Letting $\left|\mathcal{A}\right|=x$, and $z=x^{1/\beta_{\kappa}}$, we have  
\begin{equation*}
S\left(\mathcal{A},\mathcal{P},z\right)\gg\frac{x}{\log^{\kappa}x}
\end{equation*}
for pairs $\kappa$ and $\beta_{\kappa}$ listed in the table below.
\begin{center}
\vspace{0.125in}
\begin{tabular}{c}
\begin{tabular}{|c|c|c|c|c|c|c|c|c|c|}\hline
$\kappa$&$2$&$3$&$4$&$5$&$6$&$7$&$8$&$9$&$10$\\\hline
$\beta_{\kappa}$&$4.516$&$6.520$&$8.522$&$10.523$&$12.524$&$14.524$&$16.524$&$18.525$&$20.525$\\\hline
\end{tabular}
\end{tabular}
\end{center}
\end{theorem}
Thus, Selberg's sifting limit is approached rapidly from below. Indeed, although we have restricted the argument to integer $2\le\kappa\le10$, we expect that $\beta_{\kappa}\le2\kappa+19/36$ for all $\kappa$. When compared with the DHR sieves, the $\Lambda^2\Lambda^-$ sieve gives a better sifting limit $\beta_{\kappa}$ for integral $\kappa\ge 3$. The table below gives a comparison of the two sieves.
\begin{center}
Table 1. Sifting Limit Comparison\\
\vspace{0.125in}
\begin{tabular}{c}
\begin{tabular}{|c|c|c|c|c|c|c|c|c|c|}\hline
$\kappa$&2&3&4&5&6&7&8&9&10\\\hline
DHR $\beta_{\kappa}$&4.266&6.640&9.072&11.534&14.014&16.504&18.998&21.495&23.992\\\hline
$\Lambda^{2}\Lambda^{-}\ \beta_{\kappa}$&4.516&6.520&8.522&10.523&12.524&14.524&16.524&18.525&20.525\\\hline
\end{tabular}\label{table1ch2}
\end{tabular}
\end{center}
\vspace{0.125in}
More improvements are certainly possible. Recently, Sara Blight \cite[pp. 28-29]{Blight} has shown that $\beta_{2}<4.45$, $\beta_{3}<6.458$, and $\beta_{4}<8.47$. Her work features a set of weights that take into account numbers composed of up to three prime factors. These weights were suggested by Selberg as a modification to the $\Lambda^{2}\Lambda^{-}$ sieve.

One interesting application of these sieves is to almost-primes in polynomial sequences. In a forthcoming paper, the author will show that a weighted $\Lambda^{2}\Lambda^{-}$ sieve is capable of producing better results than the weighted DHR sieves when the polynomial is a product of linear irreducible factors, for example. However, the DHR sieves still perform quite well in the higher dimensional setting when the irreducible factors of the polynomial are each of a large degree, owing to the optimal nature of the DHR construction when $\kappa=1$. 

\section{Sieve Setup}
Following Selberg, we define $f':=f\ast\mu$ and let $\lambda_{d}$ be an arbitrary sequence of real numbers with the property that $\lambda_{d}=0$ if $d$ is not squarefree, or if $d>\xi$. Next, define $\zeta_{r}$ by the relation
\begin{equation}\label{E:zetarfirst}
\frac{\mu(r)\zeta_{r}}{f'(r)}=\sum_{d}\frac{\lambda_{dr}}{f(dr)}.
\end{equation}
By M\"obius inversion, we also have
\begin{equation*}
\frac{\mu(d)\lambda_{d}}{f(d)}=\sum_{r}\frac{\zeta_{dr}}{f'(dr)}.
\end{equation*}
In the classical Selberg sieve, the $\zeta_{r}$ are constant.

Assume that $\lambda_{1}\neq0$, and let $\lambda'_{d}=\frac{\lambda_{d}}{\lambda_{1}}$. Since $\lambda'_{1}=1$,
$$\sum_{\substack{n\in\mathcal{A}\\(n,P(z))=1}}1\ge\sum_{n\in\mathcal{A}}\left(1-\sum_{\substack{p|n\\p<z}}1\right)\left(\sum_{\substack{\nu|n\\\nu\mid P(z)}}\lambda'_\nu \right)^2.$$
The right-hand side can be rearranged using a well-known identity. In particular, we have
\begin{lemma}
With $\zeta_{r}$ defined as in \eqref{E:zetarfirst}, we have
\begin{align}\label{E:Sievech2}
\sum_{n\in\mathcal{A}}\left(\sum_{\substack{d|n\\d\mid P(z)}}a_d\right)\left(\sum_{\substack{\nu|n\\\nu\mid P(z)}}\lambda_\nu \right)^2=|\mathcal{A}| \fS_{\mathcal{A}}+\mathfrak{E}_{\mathcal{A}},
\end{align}
where
\begin{align}\label{E:Mainch2}
\fS_{\mathcal{A}}=\sum_{m} \sum_{\substack{d\\(d,m)=1}}\frac{\mu^2(m)}{f'(m)}\frac{a_d}{f(d)}\left( \sum_{r|d} \mu(r) \zeta_{rm} \right)^{2},
\end{align}
and
\begin{align}\label{E:Errorch2}
\mathfrak{E}_{\mathcal{A}}=\sum_{d,\nu_{1},\nu_{2}\mid P(z)}a_{d}\lambda_{\nu_{1}}\lambda_{\nu_{2}}\mathcal{R}_{\left[d,\nu_{1},\nu_{2}\right]}.
\end{align}
\end{lemma}
For our purposes, we divide both sides of this identity by $\lambda_{1}^2$ and choose
\begin{equation}\label{adch2}
a_{d}=
\begin{cases}
\phantom{+}1,& \text{if $d=1$,}\\
-1,& \text{if $d$ is prime and $d<z$,}\\
\phantom{+}0,& \text{otherwise.}
\end{cases}
\end{equation}
The identity in \eqref{E:Sievech2} distinguishes \eqref{E:Mainch2} as the main term and \eqref{E:Errorch2} as the error term for the sum. This identity is the starting point of the $\Lambda^{2}\Lambda^{-}$ method and has appeared in various forms in the works of Selberg \cite[See Section 7 on p.82]{Selberg}, Bombieri \cite[See Theorem 18 on p.65]{Bombieri}, Cojocaru and Murty \cite[See Theorem 10.1.1 on p.178]{CoM}, Greaves \cite[See Lemma 1 on p.286]{Greaves}, and others.

To produce a positive lower bound for \eqref{SAP} we will show that 
$$|\mathcal{A}| \frac{\fS_{\mathcal{A}}}{\lambda_{1}^2}+\frac{\mathfrak{E}_{\mathcal{A}}}{\lambda_{1}^2}\gtrsim \left|\mathcal{A}\right|V(z)\left(c+o\left(1\right)\right),$$
where $c$ is some small positive constant, and
$$V(z)=\prod_{p<z}\left(1-\frac{1}{f(p)}\right).$$
To begin, suppose that $\left|\mathcal{A}\right|=x$, and let $z=x^{1/u}$. It is easy to see that
$$V(z)^{-1}\ll\log^{\kappa}x.$$
Next, choosing
\begin{equation}\label{Error Control}
z\xi^2=x^{1-\varepsilon}, 
\end{equation}
and recalling \eqref{remainder}, we have
\begin{align*}
\frac{\mathfrak{E}_{\mathcal{A}}}{\lambda_{1}^2}&\ll\sum_{\substack{m<z\xi^{2}\\m\mid P(z)}}\left|\mathcal{R}_{m}\right|\sum_{\substack{d,\nu_{1},\nu_{2}\\\left[d,\nu_{1},\nu_{2}\right]=m}}1=\sum_{\substack{m<z\xi^{2}\\m\mid P(z)}}7^{\nu(m)}\left|\mathcal{R}_{m}\right|\\
&\ll\sum_{m<z\xi^{2}}\mu^2(m)7^{\nu(m)}\left|\mathcal{R}_{m}\right|\ll\frac{x}{\log^{\kappa+1}x}.
\end{align*}
Here we have used the fact that $\frac{\lambda_{d}}{\lambda_{1}}$ is bounded, which will be explained below. 

The $\zeta_{r}$ will be chosen as
$$\zeta_{r}=P\left(\frac{\log \xi/r}{\log z}\right),$$
where $P(w)$ is a polynomial that is positive for $0\le w\le u$. Therefore, 
\begin{equation}\label{B:firstlambdabound}
\lambda_{1}=\sum_{r<\xi}\frac{\zeta_{r}}{f'(r)}\le\sup_{0\le w\le u}P\left(w\right)\sum_{\substack{r<\xi\\ r\mid P(z)}}\frac{1}{f'(r)}\ll\sum_{r\mid P(z)}\frac{1}{f'(r)}=\frac{1}{V(z)}.
\end{equation}
In the case when $\zeta_{r}=1$, the $\lambda_{\nu}$ are well-understood. We will refer to this choice of $\lambda_{\nu}$ as $\widetilde{\lambda}_{\nu}$. It is known, for example, that $\left|\widetilde{\lambda}_{\nu}\right|\le\left|\widetilde{\lambda}_{1}\right|$. Since
\begin{align}\label{E:polyP}
\left|\lambda_{d}\right|\le\sup_{0\le w\le u}\left|P(w)\right|\widetilde{\lambda}_{1},
\end{align}
and
$$\lambda_{1}=\sum_{\substack{r<\xi\\r\mid P(z)}}\frac{\mu^{2}(r)}{f'(r)}P\left(\frac{\log \xi/r}{\log r}\right)\ge\inf_{0\le w\le u}P(w)\widetilde{\lambda}_{1},$$
it is clear that
\begin{equation}\label{supbound}
\frac{\left|\lambda_{\nu}\right|}{\left|\lambda_{1}\right|}\le\frac{\displaystyle\sup_{0\le w\le u}|P(w)|}{\displaystyle\inf_{0\le w\le u}|P(w)|}.
\end{equation}
It follows that the sequence
$$\lambda_{\nu}'=\frac{\lambda_{\nu}}{\lambda_{1}}$$
is bounded.

Finally, since 
$$|\mathcal{A}| \frac{\fS_{\mathcal{A}}}{\lambda_{1}^2}+\frac{\mathfrak{E}_{\mathcal{A}}}{\lambda_{1}^2}=|\mathcal{A}|V(z)\left(\frac{\mathfrak{S}_{\mathcal{A}}V(z)}{\left(\lambda_{1}V(z)\right)^{2}}+\frac{1}{|\mathcal{A}|V(z)}\frac{\mathfrak{E}_{\mathcal{A}}}{\lambda_{1}^2}\right),$$
we have
$$|\mathcal{A}| \frac{\fS_{\mathcal{A}}}{\lambda_{1}^2}+\frac{\mathfrak{E}_{\mathcal{A}}}{\lambda_{1}^2}=|\mathcal{A}|V(z)\left(\frac{\mathfrak{S}_{\mathcal{A}}V(z)}{\left(\lambda_{1}V(z)\right)^{2}}+O\left(\frac{1}{\log x}\right)\right).$$
We showed in \eqref{B:firstlambdabound} that $\lambda_{1}V(z)$ is bounded, and so our priority is in the analysis of $\mathfrak{S}_{A}V(z)$.

\section{Analysis of the Main Term}

In this section, we will treat the expression $\mathfrak{S}_{\mathcal{A}}$ occurring in the main term of the $\Lambda^{2}\Lambda^{-}$ lower bound sieve. First, let us recall that with Selberg's choice of weights $a_d$ in \eqref{adch2} we have that 
$$\fS_{\mathcal{A}}>\sum_{\substack{m<\xi\\ m|P(z)}}\frac{\mu^2(m)}{f'(m)}\zeta^2_m-\sum_{\substack{m<\xi\\ m|P(z)}}\frac{\mu^2(m)}{f'(m)} 
\sum_{p<z}\frac{1}{f(p)}\left(\zeta_m-\zeta_{pm}\right)^2,$$
upon omitting the condition that $(m,p)=1$ in \eqref{E:Mainch2}. We wish to smooth this expression using the asymptotic formulas for
\begin{equation}\label{E:G estimate}
G\left(r,z\right)=\sum_{\substack{m<r\\m\mid P(z)}}\frac{\mu^2(m)}{f'(m)}\sim\frac{j_{\kappa}\left(\frac{\log r}{\log z}\right)}{V(z)}, 
\end{equation}
and
\begin{equation}\label{E:H estimate}
H(s)=\sum_{p<s}\frac{\log p}{f(p)}\sim\kappa\log s,
\end{equation}
where $j_{\kappa}(u)$ is the continuous solution of the differential delay equation 
\begin{equation}\label{E:prop2}
uj'(u)=\kappa j(u)-\kappa j(u-1),
\end{equation}
for $u>1$, with
\begin{equation}\label{E:prop1}
j(u)=
	\begin{cases}
	   \displaystyle\frac{e^{-\gamma\kappa}}{\Gamma(\kappa+1)}u^{\kappa}, & \text{if $0<u\le1$,}\\
	    \phantom{+}0, & \text{if $u\le0$}.
	 \end{cases}
\end{equation}
We remark that if $\kappa$ is held fixed, then $j_{\kappa}(u)$ increases to $1$. Now, using Riemann-Stieltjes integration and replacing the integrators with their corresponding smooth approximations in \eqref{E:G estimate} and \eqref{E:H estimate}, we expect that  
\begin{align}\label{asym}
\fS_{\mathcal{A}}\gtrsim\frac{1}{V(z)}\int_{1}^{\xi}\zeta_{r}^{2}dj_{\kappa}\left(\frac{\log r}{\log z}\right)-\frac{\kappa}{V(z)}\int_{1}^{\xi}\int_{1}^{z}\left(\zeta_{r}-\zeta_{sr}\right)^{2}\frac{d\log s}{\log s}dj_{\kappa}\left(\frac{\log r}{\log z}\right).
\end{align}
This is indeed the case since, more specifically, if one regards $\kappa$ and $u:=\frac{\log\xi}{\log z}\ge1$ as fixed, then one has
\begin{equation}\label{E:G precise estimate}
G\left(r,z\right)=\sum_{\substack{m<r\\m\mid P(z)}}\frac{\mu^2(m)}{f'(m)}=\frac{j_{\kappa}\left(\frac{\log r}{\log z}\right)}{V(z)}\left(1+O\left(\frac{1}{\log z}\right)\right), 
\end{equation}
and
\begin{equation}\label{E:H precise estimate}
H(s)=\sum_{p<s}\frac{\log p}{f(p)}=\kappa\log s+O\left(1\right),
\end{equation}
making the error in \eqref{asym} of order at most $(V(z)\log z)^{-1}$. The formula in \eqref{E:H precise estimate} is merely our assumed density hypothesis in \eqref{loghypoth}. On the other hand, the bound in \eqref{E:G precise estimate} is a consequence of 
\begin{lemma}\label{hr74asymptoticlemma}
For any $\tau=\frac{\log r}{\log z}>0$, we have
\begin{equation*}
\frac{1}{G\left(r,z\right)}=V(z)\left(\frac{1}{j_{\kappa}(\tau)}+O\left(\frac{\tau^{2\kappa+1}}{\log z}\right)\right).
\end{equation*}
\end{lemma}
Lemma \ref{hr74asymptoticlemma} is discussed in some detail in Halberstam and Richert\cite[See Section 4 on p.197]{HR74}. Now, let us define 
$$u=\frac{\log\xi}{\log z}\ge1,$$
and 
\begin{equation}\label{E:zeta r}
\zeta_{r}=P^{*}\left(\frac{\log\xi/r}{\log z}\right),
\end{equation}
where 
\begin{equation}\label{E:P star}
P^{*}\left(\frac{\log\xi/r}{\log z}\right):=
	\begin{cases}
	 P\left(\frac{\log \xi/r}{\log z}\right) & \text{if $r<\xi$,} \\
	    \phantom{+}0 & \text{if $r\ge\xi$}.
	\end{cases}
\end{equation}
is a polynomial in the range $r<\xi$. Using these definitions simplify the integrals occurring in the analysis of $\mathfrak{S}_{\mathcal{A}}$, and making the variable change $v=\frac{\log r}{\log z}$, and $t=\frac{\log s}{\log z}$, in \eqref{asym}, we have 
$$\fS_{\mathcal{A}}\gtrsim\frac{1}{V(z)}\left(\mathcal{I}_{1}-\kappa\mathcal{I}^{*}_{2}\right),$$
where
$$\mathcal{I}_{1}=\int_{0}^{u}P^{*}\left(u-v\right)^{2}j_{\kappa}'\left(v\right)dv,$$
and
$$\mathcal{I}^{*}_{2}=\int_{0}^{u}\int_{0}^{1}\left(P^{*}\left(u-v\right)-P^{*}\left(u-v-t\right)\right)^{2}\frac{dt}{t}j_{\kappa}'\left(v\right)dv.$$
Furthermore, after making the change of variable $w=u-v$, and using \eqref{E:P star}, these integrals further simplify to
$$\mathcal{I}_{1}=\int_{0}^{u}P\left(w\right)^{2}j_{\kappa}'(u-w)dw,$$
and
$$\mathcal{I}^{*}_{2}=\int_{0}^{u}\int_{0}^{1}\left(P\left(w\right)-P^{*}\left(w-t\right)\right)^{2}\frac{dt}{t}j_{\kappa}'(u-w)dw.$$
The latter integral must be further dissected to account for the vanishing of the function $P^{*}$ in the range $w\le t\le 1$. As a result, the region in the latter integral naturally splits into three distinct pieces. After splitting the range of integration to account for this, we find that   
\begin{equation}\label{E:refined main term}
\fS_{\mathcal{A}}\gtrsim \frac{1}{V(z)}\left(\mathcal{I}_{1}-\kappa\mathcal{I}_{2}-\kappa\mathcal{I}_{3}-\kappa\mathcal{I}_{4}\right),
\end{equation}
where
\begin{align}
\mathcal{I}_{1}&=\int_{0}^{u}P\left(w\right)^{2}j_{\kappa}'\left(u-w\right)dw,\label{E:cal I1}\\
\mathcal{I}_{2}&=\int_{1}^{u}\int_{0}^{1}\left(P\left(w\right)-P\left(w-t\right)\right)^{2}\frac{dt}{t}j_{\kappa}'\left(u-w\right)dw,\label{E:cal I2}\\
\mathcal{I}_{3}&=\int_{0}^{1}\int_{0}^{w}\left(P\left(w\right)-P\left(w-t\right)\right)^{2}\frac{dt}{t}j_{\kappa}'\left(u-w\right)dw,\label{E:cal I3}\\
\mathcal{I}_{4}&=\int_{0}^{1}\int_{w}^{1}P\left(w\right)^{2}\frac{dt}{t}j_{\kappa}'\left(u-w\right)dw.\label{E:cal I4}
\end{align}
Contrary to initial appearances, the innermost integral in \eqref{E:cal I2} and \eqref{E:cal I3} does not have a singularity at $t=0$ because the constant term does not appear in the difference $P\left(w\right)-P\left(w-t\right)$. The next step is to employ a device of Grupp and Richert to evaluate these integrals. Before moving on, let us remark that if $u:=\kappa-1/3-d$, and $\kappa$ is taken sufficiently large, Selberg\cite[See pp.174-176]{Selberg} has shown that if one sets $P(w)=w+a$, one has
$$\fS_{\mathcal{A}}\gtrsim\frac{1}{V(z)}\left(-a^2+\frac{1}{2}a-\frac{(2+9d)}{18}\right)\sqrt{\frac{\kappa}{\pi}}.$$
Choosing $a$ so that $d$ is as large as possible with $-a^2+\frac{1}{2}a-\frac{2+9d}{18}>0$, we see that the optimal choice is $a=1/4$, which implies that a positive lower bound is achieved when $d<-7/72$. A slightly more complicated argument that involves a more sophisticated set of weights will give $d\le -7/72$, and this is enough to show that the sifting limit $\beta_{\kappa}\gtrsim2u+1=2\kappa+\frac{19}{36}$, upon taking $d=-7/72$. The weights that achieve this involve divisors of $n$ consisting of two and three prime factors. As the author's investigations of the use of higher degree polynomials in this problem has not met with much success, we will follow Selberg and restrict our attention to linear polynomials as well. 

\section{The $K_{n}(u,\lambda)$ Functions}

In order to evaluate the integrals arising in our sieve, we will need to decompose $j_{\kappa}'$. In his dissertation, Wheeler \cite[See Proposition 3.1.1 on p.18]{Wheeler} noted that $j_{\kappa}$, as well as its derivatives, could be decomposed into a sum of simpler functions $K_{n}(u,\lambda)$, each defined for $\lambda>-1$ and $n\ge0$. More specifically, we have
\begin{equation}\label{E:jdecom}
j_{\kappa}(u)=\frac{e^{-\kappa\gamma}}{\Gamma(\kappa+1)}\sum_{0\le n<u}(-\kappa)^{n}K_{n}(u,\kappa).
\end{equation}
The sequence of functions $K_{n}(u,\lambda)$ is defined by the equations
\begin{equation}\label{E:Kprop3}
K_{0}(u,\lambda)=u^{\lambda} \qquad u>0,\\
\end{equation}
and
\begin{equation}\label{E:1}
K_{n}(u,\lambda)=u^{\lambda}\int_{n}^{u}t^{-\lambda-1}K_{n-1}(t-1,\lambda)dt \qquad u>n\ge1.
\end{equation}
We also specify that these functions vanish if $u\le n$, and thus
\begin{equation}\label{E:Kprop2}
K_{n}(u,\lambda)=0 \qquad u\le n.
\end{equation}
To justify that the decomposition in \eqref{E:jdecom} is valid, one can verify that the expression on the right-hand side of \eqref{E:jdecom} satisfies the delay-differential equation in \eqref{E:prop2}. This follows from the observation that for $n\ge1$,
$$uK_{n}'(u,\kappa)=K_{n-1}(u-1,\kappa)+\kappa K_{n}(u,\kappa).$$
Upon separating the first term from the decomposition in \eqref{E:jdecom}, we have
\begin{align*}
&u\frac{e^{-\kappa\gamma}}{\Gamma(\kappa+1)}\sum_{0\le n<u}(-\kappa)^{n}K_{n}'(u,\kappa)\\
=&u\left(\frac{e^{-\kappa\gamma}}{\Gamma(\kappa+1)}K_{0}'(u,\kappa)+\frac{e^{-\kappa\gamma}}{\Gamma(\kappa+1)}\sum_{1\le n<u}(-\kappa)^{n}K_{n}'(u,\kappa)\right)\\
=&\kappa\frac{e^{-\kappa\gamma}}{\Gamma(\kappa+1)}K_{0}(u,\kappa)+\frac{e^{-\kappa\gamma}}{\Gamma(\kappa+1)}\sum_{1\le n<u}(-\kappa)^{n}uK_{n}'(u,\kappa)\\
=&\kappa\frac{e^{-\kappa\gamma}}{\Gamma(\kappa+1)}K_{0}(u,\kappa)+\frac{e^{-\kappa\gamma}}{\Gamma(\kappa+1)}\sum_{1\le n<u}(-\kappa)^{n}\left(K_{n-1}(u-1,\kappa)+\kappa K_{n}(u,\kappa)\right),
\end{align*}
which is
\begin{align*}
=&\kappa\frac{e^{-\kappa\gamma}}{\Gamma(\kappa+1)}\sum_{0\le n<u}(-\kappa)^{n}K_{n}(u,\kappa)-\kappa\frac{e^{-\kappa\gamma}}{\Gamma(\kappa+1)}\sum_{1\le n<u}(-\kappa)^{n-1}K_{n-1}(u-1,\kappa)\\
=&\kappa\frac{e^{-\kappa\gamma}}{\Gamma(\kappa+1)}\sum_{0\le n<u}(-\kappa)^{n}K_{n}(u,\kappa)-\kappa\frac{e^{-\kappa\gamma}}{\Gamma(\kappa+1)}\sum_{0\le n<u-1}(-\kappa)^{n}K_{n}(u-1,\kappa).
\end{align*}
Therefore, the expression occurring on the right-hand side of \eqref{E:jdecom} satisfies the same differential delay equation as the sieve function $j_{\kappa}(u)$. We will only be concerned with integral dimensions $\kappa$ throughout this discussion, and thus will focus on integral $\lambda>-1$. In fact, the most important case occurs when $\lambda=0$ and the following lemma will provide us with a useful tool to understand the cases when $\lambda\neq0$.
\begin{lemma}\label{E:DerivativeLemma}
If $\lambda\ge1$ and $n\ge0$, then
$$\frac{d}{du}K_{n}(u,\lambda)=\lambda K_{n}(u,\lambda-1)$$
\end{lemma}
\begin{proof}
We assume that $u\ge n$, for the result is obvious otherwise. Our proof is by induction on $n$; the case $n=0$ is obvious. From \eqref{E:1}, we see that 
\begin{equation}\label{E:2}
\frac{d}{du}K_{n+1}(u,\lambda)=\frac{K_{n}(u-1,\lambda)}{u}+\lambda u^{\lambda-1}\int_{n+1}^{u}\frac{K_{n}(t-1,\lambda)}{t^{\lambda+1}}dt.
\end{equation}
On the other hand, we can use \eqref{E:1} together with the inductive hypothesis and integration by parts to get
\begin{align*}
\lambda K_{n+1}(u,\lambda-1)&=u^{\lambda-1}\int_{n+1}^{u}\frac{dK_{n}(t-1,\lambda)}{t^{\lambda}}\\
&=\frac{K_{n}(u-1,\lambda)}{u}+\lambda u^{\lambda-1}\int_{n+1}^{u}\frac{K_{n}(t-1,\lambda)}{t^{\lambda+1}}dt.
\end{align*}
The desired result follows by comparing this with \eqref{E:2}.
\end{proof}
As an application of this lemma, it is easy to deduce that
\begin{equation}\label{E:jderdecom}
j_{\kappa}'(u)=\frac{e^{-\kappa\gamma}}{\Gamma(\kappa)}\sum_{0\le n<u}(-\kappa)^{n}K_{n}(u,\kappa-1),
\end{equation}
and indeed expressions for higher derivatives of $j_{\kappa}$ can be obtained with more applications of Lemma \ref{E:DerivativeLemma}, if desired. 

\section{The Case $\lambda=0$}
  
Grupp and Richert \cite{GrRi} made a close study of $K_{n}(u,0)$, obtaining useful power series representations for these functions. Their notation differs from Wheeler's, but their results can be translated easily since
$$K_{n}(u,0)=(u+1)I_{n+1}(u+1).$$
In this section and the following one, we shall write $K_{n}(u)$ in place of $K_{n}(u,0)$. We can obtain an analytic continuation of the function $K_{n}(u)$ if we define $K_{n}(z)$ by the equations 
$$K_{0}(z)=1 \qquad \Re{z}>-1,$$
and
$$K_{n}(z)=\int_{n}^{z}K_{n-1}(z-1)\frac{dt}{t} \quad \Re{z}>n-1.$$
It is easy to see that $K_{n}(z)$ is an analytic function for $\Re{z}>n-1$ and coincides with $K_{n}(u)$ for real values of $u\ge n$. Thus, the power series
\begin{equation}\label{E:Series}
K_{n}(u)=\sum_{j=0}^{\infty}b_{j}(n,c)(u-(n+c))^{j}
\end{equation}
is valid for $\left|u-(n+c)\right|<1+c$ and $u\ge n$, and any $c\ge0$. Moreover, the constant coefficients satisfy
\begin{equation}\label{first coefficient}
b_{0}(n,c)=K_{n}(n+c).
\end{equation}
Now, we have the following recursive formula for the rest of the coefficients $b_{j}(n,c)$, where $j\neq0$.
\begin{lemma}\label{coefficient lemma}
If $j\ge1$, $n\ge1$, and $c\ge0$, then
$$b_{j}(n,c)=\frac{(-1)^{j-1}}{j(n+c)^{j}}\sum_{l=0}^{j-1}(-1)^{l}b_{l}(n-1,c)(n+c)^{l}.$$
\end{lemma}
\begin{proof}
From \eqref{E:1} and \eqref{E:Series}, we obtain $K_{n}'(u)=K_{n-1}(u-1)/u$ and 
$$\sum_{j=0}^{\infty}jb_{j}(n,c)(u-(n+c))^{j-1}=\frac{1}{u}\sum_{l=0}^{\infty}b_{l}(n-1,c)(u-(n+c))^{l}.$$
If $\left|u-(n+c)\right|<1+c$ and $n\ge1$, then $\left|u-(n+c)\right|<n+c$ and
$$\frac{1}{u}=\sum_{k=0}^{\infty}\frac{(u-(n+c))^{k}(-1)^{k}}{(n+c)^{k+1}}.$$
Inserting this last equation into the previous one, we find that
$$\sum_{j=0}^{\infty}jb_{j}(n,c)(u-(n+c))^{j-1}=\sum_{k=0}^{\infty}\sum_{l=0}^{\infty}\frac{(-1)^{k}}{(n+c)^{k+1}}b_{l}(n-1,c)(u-(n+c))^{k+l}.$$
The desired result follows by equating coefficients of $(u-(n+c))^{j-1}$ on both sides.
\end{proof}
An alternative form of the recursive formula for the sequence $b_{j}(n,c)$ will also be useful for induction arguments to follow.
\begin{lemma}\label{coefficient lemma2}
If $j\ge1$, $n\ge1$, and $c\ge0$, then
$$b_{j}(n,c)=\frac{1}{j(n+c)}\left\{b_{j-1}(n-1,c)-(j-1)b_{j-1}(n,c)\right\}.$$ 
	\end{lemma}
	\begin{proof}
First, observe that from \eqref{E:Series}, since $b_{j}(n,c)$ are precisely the coefficients in the power series expansion of $K_{n}(u)$ centered about $u=n+c$, we have 
\begin{equation}\label{E:powercoeff}
b_{j}(n,c)=\left.\frac{K_{n}^{(j)}(u)}{j!}\right|_{u=n+c}.
\end{equation}
Next, from \eqref{E:1}, we see that
$$K_{n-1}^{(j-1)}(u-1)=\left(uK_{n}'(u)\right)^{(j-1)}=\sum_{l=0}^{j-1}\binom{j-1}{l}u^{(l)}K_{n}^{(j-l)}(u)=uK_{n}^{(j)}(u)+(j-1)K_{n}^{(j-1)}(u).$$
Upon dividing both sides of this equation by $(j-1)!$ and evaluating at $u=n+c$, the formula follows from \eqref{E:powercoeff}. 
	\end{proof}

Grupp and Richert \cite{GrRi} gave the useful bound
\begin{equation}\label{GRbound}
\left|b_{j}(n,c)\right|\le\frac{1}{j(1+c)^{j}},
\end{equation}
valid for $0\le c\le5$ and $j\ge2$. We will need a bound in a larger range of $c$ for our purposes. Also, we will be content to accept a slightly worse bound in exchange for a simpler proof. Thus, we prove 
\begin{lemma}\label{L:newboundforb}
For $n\ge0$, $j\ge2$, and $0\le c\le19$,
\begin{equation}\label{E:newboundforb}
\left|b_{j}(n,c)\right|\le\frac{4}{(1+c)^{j}}.
\end{equation}
\end{lemma}
\begin{proof}
The proof will proceed by induction on both $j$ and $n$. First, calculations show that
\begin{align*}
b_j(0,c)= & 
  \begin{cases} 
       1 & \text{ if $j=0$,}\\
       0 & \text{ if $j\ge 1$.}
 \end{cases}
\end{align*}
due to the simple form of $K_{0}(u)$. Using this calculation together with the recursive nature of the coefficients, we also calculate that
\begin{align*}
b_j(1,c)= &
  \begin{cases}
    \log(c+1) & \text{ if $j=0$,}\\
   \displaystyle
           \frac{(-1)^{j-1}}{j(c+1)^j} & \text{ if $j\ge 1$,}
  \end{cases}
\end{align*}
and,
\begin{align*}
b_j(2,c) = & 
   \begin{cases}
    K_2(2+c) & \text{ if $j=0$,}\\
    \displaystyle \frac{\log(c+1)}{c+2} & \text{ if $j=1$,}\\
    \displaystyle \frac{(-1)^{j-1}}{j(c+2)^j} 
           \left\{ \log(c+1) - \sum_{l=1}^{j-1} \frac{1}{l}\left(\frac{c+2}{c+1}\right)^{l} \right\}
               & \text{ if $j\ge 2$}.
      \end{cases}    
\end{align*}
The bound claimed in the lemma is therefore clear for $n=0$ and $n=1$. For the case when $n=2$, we will need to show that
$$\left|\log(c+1)-\sum_{l=1}^{j-1} \frac{1}{l}\left(\frac{c+2}{c+1}\right)^{l}\right|\le j\left(\frac{c+2}{c+1}\right)^{j}.$$ 
For one side of the inequality, we have that for $0\le c\le19$,
$$\log(c+1)-\sum_{l=1}^{j-1} \frac{1}{l}\left(\frac{c+2}{c+1}\right)^{l}\le \log(c+1)-\frac{c+2}{c+1}<2\le j\left(\frac{c+2}{c+1}\right)^j.$$
For the other side of the inequality, we must show that
$$\sum_{l=1}^{j-1} \frac{1}{l}\left(\frac{c+2}{c+1}\right)^{l}-\log(c+1)\le j\left(\frac{c+2}{c+1}\right)^{j},$$
but here, Grupp and Richert \cite[first formula below (4.6)]{GrRi} obtain the superior bound
$$\sum_{l=1}^{j-1} \frac{1}{l}\left(\frac{c+2}{c+1}\right)^{l}-\log(c+1)\le \left(\frac{c+2}{c+1}\right)^{j}.$$
Let us therefore assume that $n\ge3$ from now on. Before we can induct on both $j$ and $n$, we need to prove that the bound in \eqref{E:newboundforb} holds for $j=2$. Here, Grupp and Richert \cite[formula (2.9)]{GrRi} supply us with the useful inequality
\begin{equation}\label{E:easyKinequality}
0\le K_{n}(u)\le\frac{\log^{n}(u-n+1)}{n!}.
\end{equation}
This bound clearly holds for $n=0$. By induction, when $n\ge1$, we have
\begin{align*}
K_{n}(u)&\le\frac{1}{(n-1)!}\int_{n}^{u}\log^{n-1}(t-n+1)\frac{dt}{t}\\
&\le\frac{1}{(n-1)!}\int_{n}^{u}\log^{n-1}(t-n+1)d\log(t-n+1)=\frac{\log^{n}(u-n+1)}{n!},
\end{align*}
since $t-n+1\le t$ for $n\ge 1$. Thus, since $\log(c+1)<3$ for $0\le c\le19$, it follows from Lemma \ref{coefficient lemma2}, \eqref{E:powercoeff}, and \eqref{E:easyKinequality} that
\begin{align*}
\left|b_{2}(n,c)\right|&=\frac{1}{2(n+c)}\left|\frac{K_{n-2}(n-2+c)}{(n-1+c)}-\frac{K_{n-1}(n-1+c)}{(n+c)}\right|\\
&\le\frac{1}{2(n+c)}\max\left\{\frac{K_{n-2}(n-2+c)}{(n-1+c)},\frac{K_{n-1}(n-1+c)}{(n+c)}\right\}\\
&\le\frac{1}{2(n+c)}\max\left\{\frac{\log^{n-2}(c+1)}{(n-2)!(n-1+c)},\frac{\log^{n-1}(c+1)}{(n-1)!(n+c)}\right\}\\
&\le\frac{1}{2(1+c)^{2}}\max\left\{\frac{\log^{n-2}(c+1)}{(n-2)!},\frac{\log^{n-1}(c+1)}{(n-1)!}\right\}\\
&\le\frac{1}{2(1+c)^{2}}\max\left\{\frac{3^{n-2}}{(n-2)!},\frac{3^{n-1}}{(n-1)!}\right\}\le\frac{4}{(1+c)^2}.
\end{align*}
To complete the induction, we observe that if $j\ge3$ and $n\ge3$,
\begin{align*}
\left|b_{j}(n,c)\right|&=\left|\frac{1}{j(n+c)}\left\{b_{j-1}(n-1,c)-(j-1)b_{j-1}(n,c)\right\}\right|\\
&\le\frac{1}{j(n+c)}\left(\frac{4}{(1+c)^{j-1}}+(j-1)\frac{4}{(1+c)^{j-1}}\right)\\
&\le\frac{4}{(1+c)^{j}}.
\end{align*}
\end{proof}
If one requires a bound for $b_{j}(n,c)$ in a larger range of $c$ values, say $2\le c\le C$, one could probably replace the constant $4$ in the lemma above with $\log(C+1)$. A bound for $c\le19$ is more than enough for our purposes. Grupp and Richert \cite{GrRi} remarked that the bound in \eqref{GRbound} could be extended to hold for $0\le c\le9$, but with considerably more work.

\section{The Chain of Circles}

In the last section, many facts concerning the power series representations of $K_{n}(u)$ were assembled. This information will be especially useful when combined with an idea of Grupp and Richert, known as the Chain of Circles, or \textit{Kreiskettenverfahren}. The method is essentially analytic continuation. To begin, one defines the sequence
$$c_{\nu}=\left(\frac{3}{2}\right)^{\nu}-1,$$
and forms the corresponding sequence of power series
$$K_{n}(u;\nu)=\sum_{j=0}^{\infty}b_{j}(n,c_{\nu})\left(u-(n+c_{\nu})\right)^{j}.$$
This sequence of power series has the feature that it can be generated recursively. The power series for $K_{n}(u;\nu)$ is obtained from $K_{n}(u;\nu-1)$ since, using \eqref{first coefficient},
\begin{equation}\label{chain l=0}
b_{0}(n,c_{\nu})=K_{n}(n+c_{\nu};\nu-1),
\end{equation}
and the rest of the coefficients can be computed using Lemma \ref{coefficient lemma} or Lemma \ref{coefficient lemma2}. Thus, we have a chain of power series representations for $K_{n}(z)$ that can be continued throughout the half plane $\Re{z}>n-1$, as seen in Figure below.

\begin{center}
	\begin{tabular}{c}
		\includegraphics[scale=0.7]{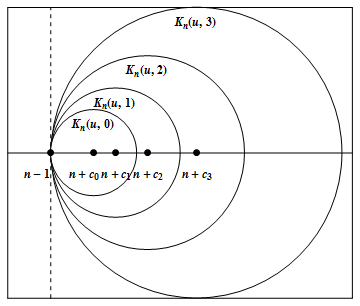}\label{fig1ch2}\\
		Figure 1. Chain of Circles
	\end{tabular}
\end{center}

Although the power series $K_{n}(u;\nu)$ is valid inside a larger interval, we will restrict the representation to the interval 
$$I_{\nu}=\left\{u:n+c_{\nu}<u\le n+c_{\nu+1}\right\}$$
to speed the convergence of the series. The sequence $c_{\nu}$, as Grupp and Richert point out, strikes a balance between the number of power series needed to cover a fixed $u$ value, and the convergence rate of each of those power series. Finally, we have obtained a useful decomposition of $K_{n}(u)$, given by 
\begin{equation}\label{chaindecomposition l=0}
K_{n}(u)=\sum_{\nu=0}^{\infty}\chi_{\nu}(u)K_{n}(u;\nu),
\end{equation}
where $\chi_{\nu}(u)$ is the characteristic function of the interval $I_{\nu}$. Now, for numerical purposes, we will truncate each of these power series to, say, $N$. Actually, for our purposes we will eventually take $N=80$. In the first circle, $K_{n}(u;0)$ will suffer only from the truncation. However, in the next circle, $K_{n}(u;1)$ will not only be truncated, but the coefficients will be approximates of the actual coefficients due to the recursive nature of $b_{0}(n,c_{1})=K_{n}(n+c_{1};0)$. Controlling the error that propagates will therefore require some work. To make our discussion more precise, let us define
\begin{equation}\label{tildaKdef l=0}
\widetilde{K}_{n}(u;\nu)=\sum_{j=0}^{N}\widetilde{b}_{j}(n,c_{\nu})\left(u-(n+c_{\nu})\right)^{j}.
\end{equation}
The coefficients $\widetilde{b}_{j}(n,c_{\nu})$ will be generated in exactly the same fashion as $b_{j}(n,c_{\nu})$ using \eqref{chain l=0} and Lemma \ref{coefficient lemma}. When $\nu=0$ we have $\widetilde{b}_{j}(n,c_{0})=b_{j}(n,c_{0})$, for $j\le N$. However, the $\widetilde{b}_{j}(n,c_{\nu})$ will be approximates of the actual coefficients $b_{j}(n,c_{\nu})$ for $\nu\ge1$ due to \eqref{chain l=0}. More specifically, we define
$$
\widetilde{b}_{j}(0,c_{0})=
\begin{cases}
	   \phantom{+} 1 & \text{if $j=0$,}\\
	   \phantom{+} 0 & \text{if $j>0$},
	 \end{cases}$$
and
$$
\widetilde{b}_{j}(n,c_{\nu})=
\begin{cases}
	   \phantom{+} \widetilde{K}_{n}(n+c_{\nu}) & \text{if $j=0$,}\\
		 \phantom{+} \displaystyle\frac{(-1)^{j-1}}{j(n+c_{\nu})^{j}}\sum_{l=0}^{j-1}\widetilde{b}_{l}(n-1,c_{\nu})(n+c_{\nu})^{l} & \text{if $0<j\le N$}\\
	   \phantom{+} 0 & \text{if $j>N$}.
	 \end{cases}$$
The following lemma of Grupp and Richert tells us that the error between the coefficients $b_{j}(n,c_{\nu})$ and $\widetilde{b}_{j}(n,c_{\nu})$, for $1\le j\le N$, can be obtained from the corresponding error when $j=0$.  
\begin{lemma}\label{gr prelim lemma}
If, for a fixed $c\ge0$, we have
$$\left|b_{0}(n,c)-\widetilde{b}_{0}(n,c)\right|\le\delta,$$
then, for $0\le j\le N$,
$$\left|b_{j}(n,c)-\widetilde{b}_{j}(n,c)\right|\le\frac{\delta}{(2+\frac{c}{2})^{j}}.$$
\end{lemma}
\begin{proof}
This is proved by induction on $n$. It is vacuously true for $n=0$ and $n=1$ since in those cases we will take $\widetilde{b}_{j}(n,c)=b_{j}(n,c)$. Now, by induction,
\begin{align*}
\left|b_{j}(n,c)-\widetilde{b}_{j}(n,c)\right|&\le\frac{1}{j(n+c)^{j}}\sum_{l=0}^{j-1}\left|b_{l}(n-1,c)-\widetilde{b}_{l}(n-1,c)\right|(n+c)^{l}\\
&\le\frac{1}{j(n+c)^{j}}\sum_{l=0}^{j-1}\frac{\delta}{(2+\frac{c}{2})^{l}}(n+c)^{l}\\
&=\frac{\delta}{(2+\frac{c}{2})^{j}}\left(\frac{1}{j}\sum_{l=0}^{j-1}\left(\frac{2+\frac{c}{2}}{n+c}\right)^{j-l}\right)\\
&\le\frac{\delta}{(2+\frac{c}{2})^{j}},
\end{align*}
since the terms in this last sum are all bounded above by one.
\end{proof}
Following Grupp and Richert, we prove
\begin{lemma}\label{prelim lemma}
If $0\le j\le N$, $\nu\ge1$, and $c_{\nu}\le19$, then
\begin{equation}\label{errorbj l=0}
\left|b_{j}(n,c_{\nu})-\widetilde{b}_{j}(n,c_{\nu})\right|\le\frac{1}{\left(2+\frac{c_{\nu}}{2}\right)^{j}}\frac{M_{\nu-1}}{2^N},
\end{equation}
where
\begin{equation}\label{initialM l=0}
M_{\nu}=4\prod_{l=0}^{\nu}\left(\frac{7+c_{l}}{3}\right)=4\prod_{l=0}^{\nu}\left(2+\frac{1}{3}\left(\frac{3}{2}\right)^{l}\right).
\end{equation} 
\end{lemma} 
\begin{proof}
We are going to use Lemma \ref{gr prelim lemma} to establish that for $\nu\ge1$,
\begin{equation}\label{errorb0 l=0}
\left|b_{0}(n,c_{\nu})-\widetilde{b}_{0}(n,c_{\nu})\right|\le\frac{M_{\nu-1}}{2^N}.
\end{equation}
The proof will proceed by induction on $\nu$. For $\nu=0$, we will take $\widetilde{b}_{j}(n,c_{0})=b_{j}(n,c_{0})$. Therefore, when $\nu=1$, we use Lemma \ref{L:newboundforb} and \eqref{chain l=0} to observe that
$$\left|b_{0}(n,c_{1})-\widetilde{b}_{0}(n,c_{1})\right|\le\sum_{j>N}\left|b_{j}(n,c_{0})\right|(c_{1}-c_{0})^{j}\le\frac{4}{2^{N}}\le\frac{4}{2^{N}}\left(\frac{7+c_{0}}{3}\right).$$
Hence, by induction, the difference
$$\left|b_{0}(n,c_{\nu})-\widetilde{b}_{0}(n,c_{\nu})\right|=\left|K_{n}(n+c_{\nu};\nu-1)-\widetilde{K}_{n}(n+c_{\nu};\nu-1)\right|$$
is at most
$$\sum_{j\le N}\left|b_{j}(n,c_{\nu-1})-\widetilde{b}_{j}(n,c_{\nu-1})\right|(c_{\nu}-c_{\nu-1})^{j}+\sum_{j>N}\left|b_{j}(n,c_{\nu-1})\right|(c_{\nu}-c_{\nu-1})^{j}.$$
Using the inductive hypothesis together with Lemma \ref{L:newboundforb}, this series is bounded by
$$\left(1+\sum_{j=0}^{\infty}\left(\frac{c_{\nu}-c_{\nu-1}}{2+\frac{c_{\nu-1}}{2}}\right)^j\right)\frac{4}{2^N}\prod_{l=0}^{\nu-2}\left(\frac{7+c_{l}}{3}\right)=\frac{4}{2^N}\prod_{l=0}^{\nu-1}\left(\frac{7+c_{l}}{3}\right).$$ 
The induction is complete, and from Lemma \ref{gr prelim lemma},
$$\left|b_{j}(n,c_{\nu})-\widetilde{b}_{j}(n,c_{\nu})\right|\le\frac{1}{\left(2+\frac{c_{\nu}}{2}\right)^{j}}\frac{4}{2^N}\prod_{l=0}^{\nu-1}\left(2+\frac{1}{3}\left(\frac{3}{2}\right)^{l}\right).$$
\end{proof}
Now that we have good control of the coefficients $\widetilde{b}_{j}(n,c_{\nu})$, we prove the following bound concerning the error between $K_{n}(u;\nu)$ and $\widetilde{K}_{n}(u;\nu)$. 
\begin{lemma}\label{L:chain error1 l=0}
If $n\ge0$, $\nu\ge0$, $c_{\nu}\le19$, and $N\ge2$, then
\begin{equation}\label{newboundforerror l=0}
\left|K_{n}(u;\nu)-\widetilde{K}_{n}(u;\nu)\right|\le\frac{ M_{\nu}}{2^{N}},
\end{equation}
where $M_{\nu}$ is as in \eqref{initialM l=0}.
\end{lemma}
\begin{proof}
The proof will proceed by induction on $\nu$. When $\nu=0$, we will take $\widetilde{b}_{j}(n,c_{0})=b_{j}(n,c_{0})$, so 
$$K_{n}(u;0)-\widetilde{K}_{n}(u;0)=\sum_{j>N}b_{j}(n,c_{0})(u-(n+c_{0}))^{j}.$$
Thus, using \eqref{E:newboundforb},
$$\left|K_{n}(u;0)-\widetilde{K}_{n}(u;0)\right|\le\sum_{j>N}\left|b_{j}(n,c_{0})\right|(c_{1}-c_{0})^{j}\le\frac{4}{2^{N}}\le\frac{M_{0}}{2^{N}}.$$
For $\nu\ge1$, we use \eqref{E:newboundforb} and Lemma \ref{prelim lemma} since $\left|K_{n}(u;\nu)-\widetilde{K}_{n}(u;\nu)\right|$ is at most 
$$\sum_{j\le N}\left|b_{j}(n,c_{\nu})-\widetilde{b}_{j}(n,c_{\nu})\right|(c_{\nu+1}-c_{\nu})^{j}+\sum_{j>N}\left|b_{j}(n,c_{\nu})\right|(c_{\nu+1}-c_{\nu})^{j},$$
which is bounded by 
$$\frac{M_{\nu-1}}{2^N}\sum_{j\le N}\left(\frac{c_{\nu+1}-c_{\nu}}{2+\frac{c_{\nu}}{2}}\right)^{j}+\frac{4}{2^N}\le\left(1+\sum_{j=0}^{\infty}\left(\frac{c_{\nu+1}-c_{\nu}}{2+\frac{c_{\nu}}{2}}\right)^{j}\right)\frac{M_{\nu-1}}{2^N}=\frac{M_{\nu}}{2^{N}}.$$
\end{proof}

\section{Generalizing to Integral $\lambda\neq0$}

When considering integral $\lambda\neq0$, one is faced with the problem of understanding the coefficients of the power series representation 
\begin{equation}\label{genseriesrep}
K_{n}(u,\lambda)=\sum_{j=0}^{\infty}b_{j}(n,c,\lambda)(u-(n+c))^{j},
\end{equation}
again valid inside $\left|u-(n+c)\right|<1+c$, by the same reasoning as in \eqref{E:Series}. The critical observation here is that repeated applications of Lemma \ref{E:DerivativeLemma} can be used to write the $b_{j}(n,c,\lambda)$ in terms of $b_{j}(n,c,0)=b_{j}(n,c)$. Thus, to generate these coefficients, one can use the fact that 
\begin{equation}\label{gen first coefficient}
b_{0}(n,c,\lambda)=K_{n}(n+c,\lambda),
\end{equation}
and, for $j\neq0$,
\begin{equation}\label{gen coefficient lemma}
b_{j}(n,c,\lambda)=\frac{\lambda}{j}b_{j-1}(n,c,\lambda-1).
\end{equation}
The analytic continuation technique of Grupp and Richert will be carried out similar to the case when $\lambda=0$. As before, these power series will be chained together to generate expansions throughout the interval $u\ge n$. Thus, one defines
$$K_{n}(u,\lambda;\nu)=\sum_{j=0}^{\infty}b_{j}(n,c_{\nu},\lambda)\left(u-(n+c_{\nu})\right)^{j},$$
each one valid inside the interval $I_{\nu}=\left\{u:n+c_{\nu}<u\le n+c_{\nu+1}\right\}$. This sequence of power series can be generated recursively. The power series for $K_{n}(u,\lambda;\nu)$ is obtained from $K_{n}(u,\lambda;\nu-1)$ since 
\begin{equation}\label{chain}
b_{0}(n,c_{\nu},\lambda)=K_{n}(n+c_{\nu},\lambda;\nu-1).
\end{equation}
This is precisely how the power series expansions are \textit{chained} together. The problem, of course, is that we will have to settle for an approximation to $K_{n}(n+c_{\nu},\lambda;\nu-1)$, as this value will be obtained by a truncated power series expansion. The series are related to the $K_{n}(u,\lambda)$ functions via the decomposition,
\begin{equation}\label{chaindecomposition}
K_{n}(u,\lambda)=\sum_{\nu=0}^{\infty}\chi_{\nu}(u)K_{n}(u,\lambda;\nu),
\end{equation}
where $\chi_{\nu}(u)$ is the characteristic function of the interval $I_{\nu}$. Of course, we make the definition $K_{n}(u,0;\nu)=K_{n}(u;\nu)$. We produce power series that represent $K_{n}(u,\lambda)$ in various intervals. We will truncate these series for numerical purposes, and hence define
\begin{equation}\label{tildaKdef}
\widetilde{K}_{n}(u,\lambda;\nu)=\sum_{j=0}^{N}\widetilde{b}_{j}(n,c_{\nu},\lambda)\left(u-(n+c_{\nu})\right)^{j}.
\end{equation}
The coefficients $\widetilde{b}_{j}(n,c_{\nu},\lambda)$ are defined by
$$\widetilde{b}_{j}(0,c_{0},\lambda)=
\begin{cases}
	   \phantom{+} 0 & \text{if $0\le j<\lambda$,}\\
	   \phantom{+} 1 & \text{if $j=\lambda$,}\\
	   \phantom{+} 0 & \text{if $j>\lambda$},
	 \end{cases}$$
and
$$\widetilde{b}_{j}(n,c_{\nu},\lambda)=
\begin{cases}
	   \phantom{+} \widetilde{K}_{n}(n+c_{\nu},\lambda-1) & \text{if $j=0$,}\\
	   \phantom{+} \displaystyle\frac{\lambda}{j}\widetilde{b}_{j-1}(n,c_{\nu},\lambda-1) & \text{if $0<j\le N$,}\\
	   \phantom{+} 0 & \text{if $j>N$}.
	 \end{cases}$$
When $\nu=0$ we have that $\widetilde{b}_{j}(n,c_{0},\lambda)=b_{j}(n,c_{0},\lambda)$, for $j\le N$. As before, the $\widetilde{b}_{j}(n,c_{\nu},\lambda)$ will be approximates of $b_{j}(n,c_{\nu},\lambda)$ for $\nu\ge1$ due to \eqref{chain}. In any case, we proceed as in \eqref{chaindecomposition} and set  
$$\widetilde{K}_{n}(u,\lambda)=\sum_{\nu=0}^{\infty}\chi_{\nu}(u)\widetilde{K}_{n}(u,\lambda;\nu).$$
The purpose of this section is to bound the error between $K_{n}(u,\lambda)$ and $\widetilde{K}_{n}(u,\lambda)$. Thus, we prove 
\begin{lemma}\label{L:chain error1}
If $n\ge0$, $\nu\ge0$, $0\le\lambda<N$, $c_{\nu}\le19$, and $N\ge2$, then
\begin{equation}\label{newboundforerror}
\left|K_{n}(u,\lambda;\nu)-\widetilde{K}_{n}(u,\lambda;\nu)\right|\le\frac{\lambda! M_{\nu,\lambda}}{2^{N-\lambda}},
\end{equation}
where $M_{\nu,0}=M_{\nu}$, and
\begin{equation}\label{recursiveM}
M_{\nu,\lambda}=\sum_{k=0}^{\nu}\left(c_{k+1}-c_{k}\right)M_{k,\lambda-1}=\frac{1}{2}\sum_{k=0}^{\nu}\left(\frac{3}{2}\right)^{k}M_{k,\lambda-1}.
\end{equation}
\end{lemma}
\begin{proof}
The proof will proceed by induction on both $\nu$ and $\lambda$. The case $\lambda=0$ has already been shown in Lemma \ref{L:chain error1 l=0}. When $\nu=0$, we will take $\widetilde{b}_{j}(n,c_{0},\lambda)=b_{j}(n,c_{0},\lambda)$, so if $0\le\lambda<N$, we can make repeated use of \eqref{gen coefficient lemma} to see that
\begin{align*}
K_{n}(u,\lambda;0)-\widetilde{K}_{n}(u,\lambda;0)&=\sum_{j>N}b_{j}(n,c_{0},\lambda)(u-(n+c_{0}))^{j}\\
&=\sum_{j>N}\frac{\lambda}{j}\cdot\frac{\lambda-1}{j-1}\cdots\frac{1}{j-\lambda+1}b_{j-\lambda}(n,c_{0},0)(u-(n+c_{0}))^{j}.
\end{align*}
Thus, using \eqref{E:newboundforb},
\begin{align*}
\left|K_{n}(u,\lambda;0)-\widetilde{K}_{n}(u,\lambda;0)\right|&\le\lambda!(c_{1}-c_{0})^{\lambda}\sum_{j>N}\left|b_{j-\lambda}(n,c_{0},0)\right|(c_{1}-c_{0})^{j-\lambda}\\
&\le\lambda!(c_{1}-c_{0})^{\lambda}\frac{4}{2^{N-\lambda}}\le\frac{\lambda!(c_{1}-c_{0})^{\lambda}M_{0,0}}{2^{N-\lambda}}=\frac{\lambda!M_{0,\lambda}}{2^{N-\lambda}}.
\end{align*}
We have shown that \eqref{newboundforerror} holds for $\lambda=0$. To prove \eqref{newboundforerror}, observe that $K_{n}(u,\lambda;\nu)-\widetilde{K}_{n}(u,\lambda;\nu)$ can be rewritten using Lemma \ref{E:DerivativeLemma}, and \eqref{tildaKdef} as
$$K_{n}(n+c_{\nu},\lambda;\nu-1)-\widetilde{K}_{n}(n+c_{\nu},\lambda;\nu-1)+\lambda\int_{n+c_{\nu}}^{u}K_{n}(t,\lambda-1;\nu)-\widetilde{\widetilde{K}}_{n}(t,\lambda-1;\nu)dt,$$
where $\widetilde{\widetilde{K}}$ is $\widetilde{K}$ with $N$ replaced by $N-1$. The first two terms above correspond to the $j=0$ term of the power series expansion. Finally, the bound in \eqref{newboundforerror} follows since
\begin{align*}
\left|K_{n}(u,\lambda;\nu)-\widetilde{K}_{n}(u,\lambda;\nu)\right|&\le\frac{\lambda!M_{\nu-1,\lambda}}{2^{N-\lambda}}+\lambda\int_{n+c_{\nu}}^{u}\frac{(\lambda-1)!M_{\nu,\lambda-1}}{2^{N-\lambda}}dt\\
&\le\frac{\lambda!}{2^{N-\lambda}}\left(M_{\nu-1,\lambda}+(c_{\nu+1}-c_{\nu})M_{\nu,\lambda-1}\right)=\frac{\lambda!M_{\nu,\lambda}}{2^{N-\lambda}}.
\end{align*}
\end{proof}
Although the presence of the $\lambda!$ term in \eqref{newboundforerror} looks menacing, we plan on taking $\lambda<10$. In addition, we will take $N$ to be much larger than $\lambda$, say $N=80$, so the error will still be well under control. In the next section, we will apply this theorem to approximate $j'_{\kappa}$.

\section{Approximating $j'_{\kappa}(u)$ in the Main Computation}

Recall the integrals
\begin{align}
\mathcal{I}_{1}&=\int_{0}^{u}P\left(w\right)^{2}j_{\kappa}'\left(u-w\right)dw,\\
\mathcal{I}_{2}&=\int_{1}^{u}\int_{0}^{1}\left(P\left(w\right)-P\left(w-t\right)\right)^{2}\frac{dt}{t}j_{\kappa}'\left(u-w\right)dw,\\
\mathcal{I}_{3}&=\int_{0}^{1}\int_{0}^{w}\left(P\left(w\right)-P\left(w-t\right)\right)^{2}\frac{dt}{t}j_{\kappa}'\left(u-w\right)dw,\\
\mathcal{I}_{4}&=\int_{0}^{1}\int_{w}^{1}P\left(w\right)^{2}\frac{dt}{t}j_{\kappa}'\left(u-w\right)dw.
\end{align}
If $\mathcal{I}_{1}-\kappa\mathcal{I}_{2}-\kappa\mathcal{I}_{3}-\kappa\mathcal{I}_{4}>0$, then a positive lower bound for $S(\mathcal{A},\mathcal{P},z)$ is obtained. To compute these integrals, define
\begin{align}
\widetilde{\mathcal{I}}_{1}&=\int_{0}^{u}P\left(w\right)^{2}\widetilde{j_{\kappa}'}\left(u-w\right)dw,\label{E:cal I1 tilda}\\
\widetilde{\mathcal{I}}_{2}&=\int_{1}^{u}\int_{0}^{1}\left(P\left(w\right)-P\left(w-t\right)\right)^{2}\frac{dt}{t}\widetilde{j_{\kappa}'}\left(u-w\right)dw,\label{E:cal I2 tilda}\\
\widetilde{\mathcal{I}}_{3}&=\int_{0}^{1}\int_{0}^{w}\left(P\left(w\right)-P\left(w-t\right)\right)^{2}\frac{dt}{t}\widetilde{j_{\kappa}'}\left(u-w\right)dw,\label{E:cal I3 tilda}\\
\widetilde{\mathcal{I}}_{4}&=\int_{0}^{1}\int_{w}^{1}P\left(w\right)^{2}\frac{dt}{t}\widetilde{j_{\kappa}'}\left(u-w\right)dw,\label{E:cal I4 tilda}
\end{align}
where
$$\widetilde{j'}_{\kappa}(u)=\frac{e^{-\kappa\gamma}}{\Gamma(\kappa)}\sum_{0\le n<u}(-\kappa)^n\widetilde{K}_{n}(u,\kappa-1).$$
Set $I=\mathcal{I}_{1}-\kappa\mathcal{I}_{2}-\kappa\mathcal{I}_{3}-\kappa\mathcal{I}_{4}$, and $\widetilde{I}=\widetilde{\mathcal{I}}_{1}-\kappa\widetilde{\mathcal{I}}_{2}-\kappa\widetilde{\mathcal{I}}_{3}-\kappa\widetilde{\mathcal{I}}_{4}$. Naturally, the integral $\widetilde{I}$ approximates $I$. The computations below are performed with $N=80$ in the definition of $\widetilde{K_{n}}(u,\kappa-1)$. The error between $\widetilde{I}$ and $I$ is bounded using Lemma \ref{L:chain error1} in the last column.

	\begin{center}
Table 2. Sifting Limit Calculations\\
\vspace{0.125in}
	\begin{tabular}{c}
\begin{tabular}{|c|c|c|c|c|c|}\hline
$\kappa$&$\beta_{\kappa}$&$u$&$a$&$\widetilde{I}$&Error\\\hline
$2$&$4.516$&$1.7581$&$0.267671$&$2.9\times10^{-5}$&$6.3\times10^{-23}$\\\hline
$3$&$6.520$&$2.7601$&$0.262761$&$5.4\times10^{-6}$&$8.6\times10^{-22}$\\\hline
$4$&$8.522$&$3.7611$&$0.260302$&$2.3\times10^{-5}$&$1.2\times10^{-20}$\\\hline
$5$&$10.523$&$4.7617$&$0.258785$&$4.5\times10^{-5}$&$2.3\times10^{-19}$\\\hline
$6$&$12.524$&$5.7621$&$0.257739$&$6.7\times10^{-5}$&$4.9\times10^{-18}$\\\hline
$7$&$14.524$&$6.7623$&$0.256929$&$2.2\times10^{-5}$&$1.2\times10^{-16}$\\\hline
$8$&$16.524$&$7.76247$&$0.256318$&$9.3\times10^{-7}$&$3.9\times10^{-15}$\\\hline
$9$&$18.525$&$8.7627$&$0.255870$&$6.5\times10^{-5}$&$1.5\times10^{-13}$\\\hline
$10$&$20.525$&$9.7628$&$0.255468$&$4.8\times10^{-5}$&$6.7\times10^{-12}$\\\hline
\end{tabular}\label{table:sieve limits}
	\end{tabular}
	\end{center}
\vspace{12pt}
These computations verify the values appearing in Table \ref{table1ch2} for $\beta_{\kappa}$ given by the $\Lambda^{2}\Lambda^{-}$ sieve. These calculations conclude the proof of Theorem \ref{T:kappabiggerthan3}.

I would like to thank my advisor, Sid Graham, for supervising this work and taking the time to carefully read through this manuscript. 
\renewcommand{\baselinestretch}{1}


\small\normalsize
\begin{thebibliography}{99}



\bibitem{Blight} S. Blight, \textit{Refinements of Selberg's Sieve}, Ph. D. thesis, Rutgers, 2010.

\bibitem{Bombieri} E. Bombieri, \emph{Le grand crible dans la th\'eorie analytique des nombres}, Ast\'erisque, No. 18., Soci\'et\'e Math\'ematique de France, Paris, 1974. MR 0371840 (51 $\#$8057).


\bibitem{CoM} A. C. Cojocaru and M. R. Murty, \emph{An introduction to sieve methods and their applications}, London Mathematical Society Student Texts, 66. Cambridge University Press, Cambridge 2006. MR 2200366 (2006k:11184).


\bibitem{DHG08} H. G. Diamond and  H. Halberstam, \emph{A higher-dimensional sieve method. With an appendix  Procedures for computing sieve functions by William F. Galway}, Cambridge Tracts in Mathematics, 177. Cambridge University Press, Cambridge, 2008. MR 2458547 (2009h:11151).



\bibitem{Greaves}
   G. Greaves, \emph{Sieves in Number Theory}, Ergebnisse der Mathematik und ihrer Grenzgebiete (3) [Results in Mathematics and Related Areas (3)], vol. 43, Springer-Verlag, Berlin, 2001. MR 1836967 (2002i:11092).
   
\bibitem{GrRi}
   F. Grupp and H.-E. Richert, \emph{The functions of the linear sieve}, J. Number Theory \textbf{22} (1986), No. 2, 208-239. MR 826952 (87f:11071).
   
   
   
\bibitem{HR74} H. Halberstam and H.-E. Richert, \emph{Sieve Methods}, Academic Press $\left[\right.$A subsidiary of Harcourt Brace Jovanovich, Publishers$\left.\right]$, London-New York, 1974, London Mathematical Society Monographs, No. 4 MR 0424730 (54 $\#$12689).









 
\bibitem{Selberg}
    A. Selberg, \emph{Collected papers, Vol. II: Lectures on Sieves} Springer-Verlag, Berlin, 
    1991. MR 1295844 (95g:01032).
    
   
\bibitem{Wheeler}
	F. S. Wheeler, \textit{On two differential-difference equations arising in analytic number theory}, Ph. D. thesis, University of Illinois at Urbana-Champaign, 1988. 
	
\end{thebibliography}
\end{document}